\documentclass[10pt]{amsart}
\usepackage{amsmath}
\usepackage{amssymb,amscd}

\usepackage[colorlinks=true, urlcolor=blue,bookmarks=true,bookmarksopen=true, citecolor=blue,hypertex]{hyperref}

\numberwithin{equation}{section}

\newtheorem{defn}{DEFINITION}[section]
\newtheorem{theorem}{Theorem}[section]

\newtheorem{example}[theorem]{EXAMPLE}

\newtheorem{lemma}[theorem]{Lemma}
\newtheorem{prop}[theorem]{Proposition}
\newtheorem{remark}[theorem]{REMARK}

\def \begineq{\begin{equation}}
\def \endeq{\end{equation}}

\def \bb{\mathbb}

\def \CC{{\bb{C}}}

\def \QQ{{\bb{Q}}}
\def \RR{{\bb{R}}}

\def \ZZ{{\bb{Z}}}

\def \({\left(}
\def \){\right)}
\def \<{\langle}
\def \>{\rangle}
\def \bar{\overline}

\def \tensor{\otimes}

\begin{document}

\title{On Quasitoric Orbifolds}

\author[M. Poddar]{Mainak Poddar}

\address{Theoretical Statistics and Mathematics Unit, Indian
Statistical Institute, 203 B. T. Road, Kolkata 700108, India}

\email{mainak@isical.ac.in}

\author[S. Sarkar]{Soumen Sarkar}

\address{Theoretical Statistics and Mathematics Unit, Indian
Statistical Institute, 203 B. T. Road, Kolkata 700108, India}

\email{soumens$_-$r@isical.ac.in}

\subjclass[2000]{57R19, 57R91}

\keywords{Orbifold, torus action, cohomology ring, stable complex structure}

\abstract Quasitoric spaces were introduced by Davis and Januskiewicz in their 1991 Duke paper. There
they extensively studied topological invariants of quasitoric manifolds.
 These manifolds are generalizations or topological counterparts of nonsingular projective
 toric varieties.
 In this article we study structures and invariants of quasitoric orbifolds. In particular, we
 discuss equivalent definitions and determine the orbifold fundamental group, rational homology groups
 and cohomology ring of a quasitoric orbifold. We determine whether any quasitoric orbifold can be the quotient
 of a smooth manifold by a finite group action or not. 
  We prove existence of stable almost complex structure
 and describe the Chen-Ruan cohomology groups of an almost complex quasitoric orbifold. 
\endabstract

\maketitle

\section{Introduction}\label{intro}

 Quasitoric spaces were introduced by Davis and Januskiewicz in \cite{[DJ]} where  topological
 invariants of
 quasitoric manifolds were extensively studied. The term {\it quasitoric} however first appeared
 in the survey \cite{[BP]} which is also a good reference for many interesting developments and
 applications.
  Quasitoric manifolds enjoy many cohomological
 properties of nonsingular toric varieties but they do not necessarily have algebraic or
  complex structure;
 see Civan \cite{[Civ]} for some interesting examples.
 Hence these properties of
 toric varieties are not contingent upon such structures but a consequence of the torus action.

  In this article we study topological invariants and stable almost complex structure
 on  quasitoric orbifolds. In particular, we
 discuss equivalent definitions and determine the orbifold fundamental group, rational homology groups
 and cohomology ring of a quasitoric orbifold. We prove existence of stable almost complex structure
 and describe the Chen-Ruan cohomology groups of an almost complex quasitoric orbifold.
 Some of these results are analogues of well known facts about complete simplicial toric varieties in algebraic
 geometry. However in the tradition of Davis and Januskiewicz, our proofs are purely topological.
 In the sequel we will study existence of almost complex structure and
 almost complex morphisms.  We expect this to be an interesting category.

 A quasitoric manifold $X^{2n}$ may be defined as an even dimensional smooth
  manifold with  a
 locally standard action of the compact torus $T^n = U(1)^n$
  such that the orbit space
 has the structure of an $n-$dimensional polytope. Locally standard means that locally the
 action is $\theta$--equivariantly diffeomorphic to the standard action of
 $T^n$ on $\CC^n$, where $\theta$ is an automorphism of $T^n$. That is, every point
 $x \in X$ has a $T^n$--invariant neighborhood
 $V_x$ and a diffeomorphism $h : V_x \to U$, where $U$ is an  open set in $\CC^n$ invariant under the
 standard action of $T^n$, and an automorphism $\theta_x : T^n \to T^n$ such that $ h (t\cdot y) = \theta_x (t) \cdot h(y)$ for
  all $y \in V_x$. We point out that there are even
 dimensional compact manifolds with locally standard $T^n$ action for which the orbit space is
 not a homology polytope; see Example 4.3 of \cite{[MP]}.

  The standard action of $U(1)^n$ on $\CC^n$
  has orbit space $\RR^n_+ =\{(r_1, \ldots, r_n) \in \RR^n \, | \, r_i \ge 0 \}$.
 The orbit space of the torus action on a quasitoric manifold is therefore a simple
 polytope in $\RR^n$; that is, exactly $n$ facets meet at each vertex.
 This leads to a combinatorial model where a quasitoric manifold is defined as the quotient
  of a trivial torus  bundle $P^n \times T^n$  on a polytope  $P^n$ by
  the action of certain
 torus subgroups of $T^n$
 on the fibers over the faces of the polytope.
  The dimension of the isotropy
 torus subgroup over
  the relative interior of a face matches the codimension of the face.

  More precisely, $T^n$ can be identified with $(\ZZ^n \otimes_{\ZZ} \RR)/\ZZ^n$.
   The isotropy subgroup of each facet $F_i$ is a circle
 subgroup corresponding to a primitive vector $\lambda_i$ in $\ZZ^n$, which is determined up to choice of
 sign. The vector $\lambda_i$ is called a characteristic vector.
  The isotropy subgroup corresponding to a face which is the intersection of facets $F_{i_1}, \ldots, F_{i_k}$
 is the subtorus of $T^n$ corresponding to the subgroup of $\ZZ^n$ generated by $\lambda_{i_1}, \ldots, \lambda_{i_k}$.
 To ensure smoothness, it has to be assumed that for every vertex of the polytope the corresponding collection
  of $n$ characteristic vectors forms a basis of $\ZZ^n$ over $\ZZ$. It turns out that different choices of
  signs of the $\lambda_i$ correspond to different stable almost complex structures on the quasitoric manifold,
   see \cite{[BR]}. A characteristic vector with a definite choice of sign is
  called a {\it dicharacteristic vector}. A quasitoric manifold whose characteristic vectors have been assigned
  definite signs is called {\it omnioriented}. We will apply the same terminology in the case of orbifolds.
  Throughout, we will denote an orbifold by a bold upper-case letter and its underlying topological
  space by the same letter in normal font.

  Our first  definition of a quasitoric orbifold is constructive and
  will readily yield a differentiable orbifold atlas. Namely,
  the underlying topological space $X$ of an $n-$dimensional quasitoric orbifold $\bf X$
  is defined to be the quotient of $P^n \times T^n$ by the action of $k$--dimensional tori on the fibers
  over codimension $k$ faces of $P^n$, via some finite covering homomorphisms onto subtori of $T^n$.
  The precise definition is given in section \ref{QO}.
   The implication for the
  characteristic vectors is that they need no longer be primitive, and the collection
  of characteristic vectors corresponding to any vertex of $P^n$ need not form
  a basis of $\ZZ^n$ over $\ZZ$, but should be $\ZZ$--linearly independent.

   We follow this up
   with an equivalent axiomatic definition of a (differentiable) quasitoric orbifold akin
    to the definition
   of a quasitoric manifold via locally standard action. We also give a classification result,
   Lemma \ref{classi}.
   Later, in section \ref{almostc}, it will be evident that our definitions of a quasitoric
   orbifold are more general
    than the original definition in \cite{[DJ]}, as the quotient ${\mathcal Z}(P)/ T_K $
 of a higher dimensional  manifold by a smooth torus action.  Note that Hattori and Masuda
  \cite{[HM]} have introduced an even more
 general class of spaces called torus orbifolds, relaxing the constraint of local standardness
 on the $T^n$ action.

 Since we restrict the orbit space $P$ to be compact, quasitoric orbifolds are compact by
 definition. However this restriction is made mainly to conform to established terminology and
 state the results in a convenient form.
 We occasionally take the liberty of using combinatorial model with noncompact base space $P$
 like at the beginning of section \ref{almostc}.

 A differentiable orbifold $\bf X$ is called a {\it global quotient} if it is diffeomorphic
 as an orbifold to the quotient orbifold $[M/G]$, where $M$ is a smooth manifold and $G$ is a
 finite group acting smoothly on $M$. It is an interesting problem to decide whether a given
 orbifold is a global quotient or not.  In section \ref{ofg}
 we solve the problem completely for quasitoric orbifolds
 by determining their {\it orbifold fundamental group} and {\it orbifold universal cover}. These invariants
 were introduced by Thurston \cite{[Th]}.

 In section \ref{hom} we compute the homology of quasitoric orbifolds with coefficients in
 $\QQ$. We need to generalize the notion of $CW-$complex a little bit for this purpose. In
 section \ref{cohom} we compute the rational cohomology ring of a quasitoric orbifold and
 show that it is isomorphic to a quotient of the Stanley-Reisner face ring of the base polytope
 $P$. These results are analogous to similar formulae for simplicial toric varieties. Our
 proofs are adaptations of the proofs in \cite{[DJ]} for quasitoric manifolds and are purely
 topological.

 In section \ref{almostc} we show the existence of a stable almost complex structure on a quasitoric
 orbifold corresponding to any given omniorientation,
 following the work of Buchstaber and Ray \cite{[BR]} in the manifold case. The universal
 orbifold cover of the quasitoric orbifold is used here. As in the manifold case, we show that
 the cohomology ring is generated by the first Chern classes of some complex rank one orbifold
 vector bundles, canonically associated to facets of $P^n$. We compute the top Chern number of
 an omnioriented quasitoric orbifold. We give a necessary condition for existence of torus invariant
 almost complex structure. Whether this condition is also sufficient remains open. Finally we compute
 the Chen-Ruan cohomology groups of an almost complex quasitoric orbifold. These will be used in the 
 sequel.

 We refer the reader to \cite{[ALR]} and references therein for definitions and facts concerning orbifolds.
  The reader may also consult \cite{[MM]} for an excellent exposition of the foundations of the theory of
  (reduced) differentiable orbifolds.

\section{Definition and orbifold structure}\label{QO}

 For any $\ZZ-$module $L$ denote
$L \tensor_{\ZZ} \RR$ by $L_R$.
Let $N$ be a free $\ZZ-$module of rank $n$.
The quotient $T_N = N_R /N$ is a compact $n$--dimensional torus.
Suppose $M$ is a free submodule of $N$ of rank $m$. Let $T_M$ denote
 the torus $M_R/M$. Let $j:M_R \to N_R $ and $j_{\ast}: T_M  \to N_R/M$ be the natural inclusions.
The inclusion $i: M \to N$ induces a homomorphism
$i_{\ast}: N_R/M \to N_R/N =T_N$ defined by $i_{\ast}(a + M) = a + N$ on cosets.
Denote the composition $i_{\ast} \circ j_{\ast}: T_M \to T_N$ by $\xi_M$.
$Ker(i_{\ast}) \simeq N/M$. If $m=n$, then $j_{\ast}$ is identity and $i_{\ast}$ is surjective.
In this case $\xi_M : T_M \to T_N $ is a surjective group homomorphism with kernel $G_M =N/M $,
 a finite abelian group.

\subsection{Definition by construction}
A $2n-$dimensional quasitoric orbifold may be constructed from the following data:
a simple polytope $P$ of dimension $n$ with facets $F_i$ indexed by $I= \{ 1, \ldots, m \}$,
 a free $\ZZ-$module $N$
of rank $n$, an assignment of a vector $\lambda_i$ in $N$ to each facet $F_i$ of $P$
 such that whenever
$F_{i_1}\cap \ldots \cap F_{i_k} \neq \emptyset$   the corresponding vectors $\lambda_{i_1}, \ldots, \lambda_{i_k}$ are linearly
independent over $\ZZ$. These data will be referred to as a {\it combinatorial model} and abbreviated as
 $(P,N,\{\lambda_i\})$. The vector $\lambda_i$ is called the dicharacteristic vector corresponding to the $i$-th facet.

 Each face $F$ of $P$ of codimension $k \ge 1$ is the intersection of a unique set of
 $k$ facets $F_{i_1}, \ldots, F_{i_k}$. Let $I(F) = \{i_1, \ldots, i_k \} \subset I$.
 Let $N(F)$ denote the submodule of $N$ generated by the characteristic vectors
 $\{ \lambda_j: j \in I(F) \}$. $T_{N(F)} = N(F)_R /N(F)$  is a torus of dimension $k$.
 We will adopt the convention that $T_{N(P)} = 1$.

 Define an equivalence relation $\sim$ on the product $P \times T_N$  by
 $(p,t) \sim (q,s)$ if $p=q$ and $s^{-1}t$ belongs to the image  of the map
  $\xi_{N(F)}: T_{N(F)} \to T_N$  where $F$ is the unique face of $P$ whose  relative interior
  contains $p$. Let $X = P \times
 T_N / \sim $ be the quotient space. Let $\mathfrak{q} : P \times T_N \to X$ denote the quotient map. Then
 $X$ is a $T_N$--space and let $\pi: X \to P $ defined by $\pi([p,t]^{\sim}) = p$
  be the associated map to the orbit space $P$. The space $X$ has the structure of an orbifold,
  which we explain next.

  Pick open
 neighborhoods $U_v$ of the vertices $v$ of $P$  such that $U_v$ is the complement in $P$ of all facets
 that do not contain $v$.
  Let $X_v = \pi^{-1}(U_v) =
  U_v \times T_N / \sim $.
 For a face $F$ of $P$ containing $v$ the inclusion $\{ \lambda_i:
 i \in I(F)\}$ in $\{ \lambda_i: i \in I(v)\}$ induces an inclusion of $N(F)$ in $N(v)$ whose image will be
 denoted by $N(v,F)$. Since $\{ \lambda_i: i \in I(F)\}$ extends to a basis $\{ \lambda_i:
 i \in I(v)\}$ of $N(v)$, the natural map from the torus $T_{N(v,F)} = N(v,F)_R / N(v,F)$ to $T_{N(v)}=
 N(v)_R /N(v)$ defined by $a + N(v,F) \mapsto a + N(v)$ is an injection. We will identify
 its image with $T_{N(v,F)}$. Denote the canonical isomorphism $T_{N(F)} \to T_{N(v,F)}$ by
 $i(v,F)$.

 Define an equivalence relation $\sim_v$ on $U_v \times T_{N(v)}$ by $(p,t)\sim_v (q,s)$ if
 $p=q$ and $s^{-1}t \in T_{N(v,F)}$ where $F$ is the face whose relative interior contains $p$.
 Then $W_v= U_v \times T_{N(v)}/ \sim_v$
  is $\theta-$equivariantly diffeomorphic to an open ball in $\CC^n$ where $\theta: T_{N(v)} \to U(1)^n$ is
  an isomorphism, see \cite{[DJ]}.
 Note that the map $\xi_{N(F)}$ factors as $\xi_{N(F)} = \xi_{N(v)} \circ  i(v,F)$. Since
 $i(v,F)$ is an isomorphism, $t \in T_{N(v,F)}$ if and only if $\xi_{N(v)} (t) \in {\rm image \,}
 \xi_{N(F)}$. Hence
  the map $\xi_{N(v)} : T_{N(v)} \to T_N$ induces a map $\xi_v: W_v \to X_v$ defined by
   $\xi_v([(p,t)]^{\sim_v}) = [(p,\xi_{N(v)}(t)) ]^{\sim}$ on equivalence classes.
    $G_v = N/N(v)$, the kernel of $\xi_{N(v)}$, is a finite
  subgroup of $T_{N(v)}$ and therefore has a natural smooth, free action on $T_{N(v)}$
  induced by the group operation.  This induces smooth action of $G_v$ on
  $W_v$. This action is not free in general. Since $T_N \cong T_{N(v)}/G_v $,  $X_v$
 is homeomorphic to the quotient space $W_v/G_v$. $(W_v, G_v, \xi_v )$ is an orbifold chart on $X_v$.
 To show the compatibility of these charts as $v$ varies, we introduce some additional charts.

 For any proper face $E$ of dimension $k \ge 1$ define $U_E = \bigcap U_v $, where the intersection
  is over all vertices $v$ that
 belong to $E$. Let $X_E = \pi^{-1}(U_E)$. For a face $F$ containing $E$ there is an
 injective homomorphism $T_{N(F)} \to T_{N(E)}$ whose image we denote by $T_{N(E,F)}$.

 Let \begin{equation}\label{GE}
 N^*(E)= (N(E) \otimes_{\ZZ} \QQ) \cap N  \quad {\rm and} \quad G_E = N^*(E)/N(E).
 \end{equation}
  $G_E$ is a finite group.
   Let $\xi_{*,E} : T_{N(E)} \to T_{N^*(E)}$ be the natural homomorphism.
 $\xi_{*,E} $ has kernel $G_E$.
 Denote the quotient $N/N^*(E)$ by $N^{\perp}(E)$.
 It is a free $\ZZ-$module and $N \cong N^*(E) \oplus N^{\perp}(E)$. Fixing a choice of this
 isomorphism (or fixing an inner product on N) we may regard $N^{\perp}(E)$ as a submodule of $N$.
  Consequently $T_N = T_{N^*(E)}\times T_{N^{\perp}(E)} $.

  Define an equivalence relation  $\sim_E$ on
 $U_E \times T_{N(E)} \times T_{N^{\perp}(E)} $ by
  $(p_1,t_1,s_1) \sim_E (p_2,t_2, s_2) $ if $p_1=p_2$, $s_1 =s_2$ and $t_{2}^{-1} t_1 \in T_{N(E,F)} $
  where $F$ is the face whose relative interior contains $p_1$.
  Let $W_E = U_E \times T_{N(E)} \times T_{N^{\perp}(E)} / \sim_E$. It is diffeomorphic to $\CC^{n-k}
  \times (\CC^*)^k$. There is a natural  map $\xi_{E}: W_E \to X_E $ induced by
  $\xi_{*,E} : T_{N(E)} \to T_{N^*(E)}$ and  the identity maps on $U_E$ and $T_{N^{\perp}(E)}$.
  $(W_E,G_E, \xi_{E} )$ is an orbifold chart on $X_E$.

  Given $E$, fix a vertex $v$ of $P$ contained in $E$. $N(v) = N(E) \oplus M$ where
  $M$ is the free submodule of $N(v)$ generated by the dicharacteristic vectors
  $\lambda_j$ such that $j \in I(v)-I(E)$. Consequently $T_{N(v)}= T_{N(E)} \times T_M$.
  We can, without loss of generality, assume that $M \subset N^{\perp}(E)$. Thus we have
  a covering homomorphism $T_M \to T_{N^{\perp}(E)}$.
  For a point $x= [p,t,s] \in X_E$, choose a small neighborhood $B$ of $s$ in $T_{N^{\perp}(E)}$
  such that $B$ lifts to $T_M$. Choose any such lift and denote it by $l: B \to T_M$.
  Let $W_x= U_E \times T_{N(E)} \times B /\sim_E$. $(W_x, G_E, \xi_E)$ is an orbifold chart on
  a neighborhood of $x$, and it is induced by $(W_E, G_E, \xi_E)$. The natural  map $W_x \hookrightarrow W_v$
  induced by the map $l$ and the identification $T_{N(v)} = T_{N(E)} \oplus T_M$, and the natural
  injective homomorphism $G_E \hookrightarrow G_v$  induce an injection (also called embedding) of
  orbifold charts
  $ (W_x, G_E, \xi_E) \to (W_v, G_v, \xi_v) $.

 %
  The existence of these injections shows
  that the charts $\{(W_v, G_v,\xi_v): v {\rm \, any \, vertex\, of\,} P \} $ are compatible
  and form part of a maximal $2n-$dimensional orbifold atlas $\mathbf{A} $ for $X$.
    We denote the pair $\{ X, \mathbf{A} \}$ by $\mathbf{X}$.
  We say that $\mathbf{X}$ is the quasitoric orbifold associated to the
  combinatorial model $(P,N,\{\lambda_i\})$.

 \begin{remark}\label{dichar}
 Note that the orbifold $\mathbf{X}$ is reduced, that is, the group in each chart has
 effective action. Also note that changing the sign of a dicharacteristic vector gives rise to
 a diffeomorphic orbifold.
\end{remark}

 Recall that for any point $x$ in an orbifold, the isotropy subgroup $G_x$ is the stabilizer
 of $x$ in some orbifold chart around $x$. It is well defined up to isomorphism. We recall the
 following definition for future reference.
 \begin{defn}\label{singpt}
 A point $x \in X $ is called a smooth
 point if $G_x$ is trivial, otherwise $x$ is called singular.
 \end{defn}
  In the case of a quasitoric orbifold $\bf X$,
 for any $x \in X$, $\pi(x)$ belongs to the relative interior of a uniquely determined face $E^x$ of $P$.
 The isotropy group $G_x = G_{E^x}$ (see \eqref{GE}). We adopt the convention that $G_P = 1$.

 \begin{defn}\label{prim}
 A quasitoric orbifold is called primitive if all its characteristic vectors are primitive.
 \end{defn}

 Note that in a primitive quasitoric orbifold the local group actions are devoid of complex
 reflections (that is maps which have $1$ as an eigenvalue with multiplicity $n-1$) and
 the classification thorem of \cite{[Pr]} for germs of complex orbifold singularities applies.

 \subsection{Axiomatic definition} Analyzing the structure of the quasitoric orbifold associated to a
 combinatorial model, we make the following axiomatic definition. This is a generalization of the
 axiomatic definition of a quasitoric manifold using the notion of locally standard action, as
 mentioned in the introduction.

  \begin{defn} A $2n$--dimensional quasitoric orbifold $\bf Y$ is an orbifold whose underlying
  topological space $Y$ has a $T_N$ action, where $N$ is a fixed free
 $\ZZ$--module of rank $n$, such that the orbit space is (homeomorphic to)
  a simple $n-$dimensional polytope $P$. Denote the projection map from $Y$ to $P$
  by $\pi: Y \to P$.
  Furthermore every point $x \in Y$ has
\begin{itemize}
\item[A1)]  a $T_N$--invariant neighborhood $ V$,
\item[A2)] an associated free  $\ZZ-$module  $M$ of rank $n$ with an
  isomorphism $\theta: T_M \to U(1)^n$ and an injective module homomorphism $\iota: M \to N$
  which induces a surjective covering homomorphism $\xi_M : T_M \to T_N $,
\item[A3)]   an orbifold chart
   $(W, G, \xi)$ over $V$ where $W$ is $\theta-$equivariantly diffeomorphic to an
    open set in $\CC^n$,  $G = {\ker} \xi_M $ and $\xi: W \to V$ is an equivariant map i.e.
   $\xi(t\cdot y)= \xi_M(t)\cdot \xi(y)$ inducing a homeomorphism between $W/G$ and $V$.
\end{itemize}
\end{defn}

It is obvious that a quasitoric orbifold defined constructively from a combinatorial model
satisfies the axiomatic definition. We now demonstrate that a quasitoric orbifold defined
axiomatically is associated to a combinatorial model. Take any facet $F$ of $P$ and let $F^0$
be its relative interior. By the characterization of local charts in A3), the isotropy group
of the $T_N$ action at any point $x$ in $\pi^{-1}(F^0)$ is a locally constant circle subgroup
of $T_N$. It is the image under $\xi_M$ of a circle subgroup of $T_M$.
 Thus it determines a locally constant vector, up to choice of sign, $\lambda$ in $N$.
 Since $\pi^{-1}(F^0)$ is connected, we get a characteristic vector $\lambda$,
 unique up to sign, for each facet of $P$. That the characteristic vectors corresponding
 to all facets of $P$ which meet at a vertex are linearly independent follows from the
 fact that their preimages under the appropriate $\iota$ form a basis of $M$. Thus we
 recover a combinatorial model $(P,N,\{\lambda_i\})$ starting from $\bf Y$.

 Let $\bf X$ be the quasitoric orbifold obtained from $(P,N,\{\lambda_i\})$ by the
 construction in the previous subsection. We need to show that $\bf X$ and $\bf Y$
 are diffeomorphic orbifolds. The hard part is to show the existence of
 $T_N$--equivariant a continuous map from $X \to Y$. This can be done following
 Lemma 1.4 of \cite{[DJ]}. The idea is to stratify $\bf Y$ according to
 {\it normal orbit type}, see Davis \cite{[Dav]}. Here we need to use the fact that the
 orbifold $\bf Y$ being reduced, is the quotient of a compact smooth manifold by
 the foliated action of a compact Lie group.
  Then one can {\it blow up} (see \cite{[Dav]}) the singular strata of $Y$ to get a
  manifold $\widehat Y$ equivariantly diffeomorphic to $T_N \times P$. One has to
  modify the arguments of Davis slightly in the orbifold case. The important thing
  is that by the differentiable slice theorem each singular stratum has a neighborhood
  diffeomorphic to its orbifold normal bundle, and is thus equipped  with a fiberwise linear
   structure so that the constructions of Davis go through. Finally there is  a collapsing
   map $\widehat{Y} \to Y$ and by composition with the above diffeomorphism a map
     $ T_N \times P \to Y $.
   It is easily checked that this map induces a continuous equivariant map $X \to Y $.

\begin{defn} Let $\bf X_1$ and $\bf X_2$ be quasitoric orbifolds whose associated
base polytope $P^n$ and free $\ZZ-$module $N$ are identical. Let $\theta $ be an
automorphism of $T_N$.
 A map ${\bf f} : {\bf X_1}\to  {\bf X_2}$ of quasitoric orbifolds is called a
$\theta-$equivariant diffeomorphism if $\bf f$ is an diffeomorphism of orbifolds and
the induced map on underlying spaces $f: X_1 \to X_2 $ satisfies
$f(t\cdot x)= \theta(t)\cdot f(x)$ for all $x \in X_1$.
\end{defn}

 Two $\theta-$equivariant diffeomorphisms $\bf f $ and $\bf g$ are said to be {\it equivalent}
 if there exists equivariant diffeomorphisms ${\bf h_i}: {\bf X_i} \to {\bf X_i}$, $i = 1,\,2$,
 such that $ {\bf g} \circ {\bf h_1} = {\bf h_2} \circ {\bf f}$.
We also define, for $\theta$ as above, the $\theta-${\it translation} of a combinatorial
 model $(P,N, \{ \lambda_i \})$ to be
the combinatorial model $(P,N, \{ \theta(\lambda_i) \}) $. The following lemma classifies
quasitoric orbifolds over a fixed polytope up to $\theta-$equivariant diffeomorphism.

\begin{lemma}\label{classi} For any automorphism $\theta$ of $T_N$, the assignment of combinatorial model
defines a bijection between equivalence classes of $\theta-$equivariant diffeomorphisms of
quasitoric orbifolds and $\theta-$translations of combinatorial models.
\end{lemma}

\begin{proof} Proof is similar to Proposition 2.6 of \cite{[BR]}. Note that the existence of
a section $s : P \to Y$ for an axiomatically defined quasitoric orbifold $\bf Y$ follows from the
 blow up construction above.
\end{proof}

\subsection{Characteristic subspaces}\label{chars}
 Of special importance are certain $T_N$--invariant subspaces of $X$ corresponding to the faces of the polytope $P$.
 If $F$ is a face of $P$ of codimension $k$, then define $X(F) := \pi^{-1}(F)$. With subspace topology,
 $X(F)$ is a quasitoric
 orbifold of dimension $2n-2k$. Recall that  $N^{*}(F) = ( N(F) \tensor_{\ZZ} \QQ ) \cap N$ and
 $N^{\perp}(F) = N/N^*(F)$. Let $\varrho_F : N \to N^{\perp}(F)$ be the projection homomorphism.
 Let $J(F)\subset I $ be the index set of facets of $P$, other than $F$ in case $k=1$, that intersect
 $F$. Note that $J(F)$ indexes the set of facets of the $n-k$ dimensional polytope $F$. The combinatorial
 model for $X(F)$ is given by $(F, N^{\perp}(F), \{\varrho_F(\lambda_i)| i \in J(F)\})$.
 $X(F)$ is called a {\it characteristic subspace} of $X$, if $F$ is a facet of $P$.



\section{Orbifold fundamental group}\label{ofg}

A covering orbifold or orbifold cover
of an $n-$dimensional orbifold $\mathbf{Z}$ is a smooth map of orbifolds  $\mathbf{p}: \mathbf{Y} \to \mathbf{Z}$
whose associated continuous  map $p: Y \to Z$ between underlying spaces satisfies the following condition:
 Each point $z \in Z$ has a
neighborhood $U \cong V/\Gamma$ with $V$ homeomorphic to a connected open set in $\RR^n$,
for which each component
$W_i$ of $p^{-1}(U)$ is homeomorphic to $V/\Gamma_i$ for some subgroup $\Gamma_i \subset \Gamma$
such that the natural map $p_i: V/\Gamma_i \to V/\Gamma$ corresponds to the restriction of $p$ on
 $W_i$.

 Given an orbifold cover $\bf{p}: \mathbf{Y} \to \mathbf{Z}$ a diffeomorphism
 $\bf{h}: \mathbf{Y} \to \mathbf{Y}$ is called a deck transformation if $\bf{p}\circ \bf{h} =
 \bf{p}$.
 An orbifold cover $\bf{p}: \mathbf{Y} \to \mathbf{Z}$ is called a universal orbifold cover of $\mathbf{Z}$ if
 given any orbifold cover $\bf{p}_1:\mathbf{W} \to \mathbf{Z}$, there exists an orbifold cover
 $\bf{p}_2: \mathbf{Y} \to \mathbf{W} $ such that $\bf{p} = \bf{p}_1 \circ \bf{p}_2$.
  Every orbifold has a universal orbifold
 cover which is unique up to diffeomorphism, see \cite{[Th]}. 
 The corresponding group of deck transformations is called the
 orbifold fundamental group of $\mathbf{Z}$ and denoted $\pi_1^{\rm orb} (\mathbf{Z})$.

Suppose $\mathbf{Z}= [Y/G] $ where $Y$ is a manifold and $G$ is a finite group.
Then the following short exact sequence holds.
\begin{equation}
 1 \to \pi_1 (Y) \to \pi_1^{\rm orb} (\mathbf{Z}) \to G \to 1
\end{equation}

This implies that an orbifold $\mathbf{Z}$ can not be a global quotient if
$\pi_1^{\rm orb} (\mathbf{Z}) $  is trivial, unless $\mathbf{Z}$ is itself a
manifold.

 We first give a cononical construction of a quasitoric orbifold cover $\bf{O}$ for any
 given quasitoric orbifold ${\bf X}$. We will prove later that $\bf{O} $ is the universal
 orbifold cover of $\bf{X}$.

\begin{defn}\label{}
 Let  $\widehat{N}$ be  the submodule of $N$ generated by the characteristic vectors of ${\bf X}$.
 Let $\widehat{\lambda}_i$ denote the characteristic
vector $\lambda_i$ as an element of $\widehat{N}$. Let $\mathbf{O}$ be the quasitoric
orbifold associated to the combinatorial model $(P, \widehat{N}, \{ \widehat{\lambda}_i \} )$.
Denote the corresponding equivalence relation by $\widehat{\sim}$ so that the underlying topological
 space of $\mathbf{O}$ is $O = P \times T_{\widehat N}/\widehat{\sim} $. Denote the quotient map
 $P \times T_{\widehat N} \to O$ by $\widehat{\pi}$.
\end{defn}

 \begin{prop}\label{pouc}
The quasitoric orbifold ${\bf O}$ is an orbifold cover of the quasitoric orbifold ${\bf X}$
with deck group $N/\widehat{N}$.
\end{prop}

\begin{proof}
The inclusion $\iota: \widehat{N} \hookrightarrow N$ induces a surjective group homomorphism
$\iota_{\ast}: T_{\widehat N}= (\widehat{N}\otimes \RR)/\widehat{N} \to T_N = (N\otimes \RR)/N$
with kernel $ N/\widehat{N}$.
In fact for any face $F$ of $P$ we have commuting diagram

  \begin{equation}\label{cd1}
\begin{CD}
T_{\widehat{N}(F)} @>\xi_{\widehat{N}(F)}>> T_{\widehat N} \\
@V\iota_{_0}VV @V\iota_{\ast}VV \\
 T_{N(F)} @>\xi_{N(F)}>> T_N
\end{CD}
\end{equation}
where $\widehat{N}(F)$ is $N(F)$ viewed as a sublattice of $\widehat{N}$ and  $\iota_{_0}$ is an
 isomorphism induced by $\iota$. Thus there is an induced surjective map
 \begin{equation}\label{eqfiber}
  \iota_1: T_{\widehat N}/{\rm im}(\xi_{\widehat{N}(F)}) \to T_N/{\rm im}(\xi_{N(F)}).
  \end{equation}

  We obtain a torus equivariant
 map $f: O \to X $ defined fiberwise by \eqref{eqfiber}, that is, for any point $q\in P$
  belonging to the relative interior of the face $F$, the restriction of $f: \widehat{\pi}^{-1}(q)
 \to \pi^{-1}(q)$ matches $\iota_1$.

 The map $f$ lifts to a smooth map of orbifolds $\bf{f}: \bf{O} \to \bf{X}$. Consider orbifold charts
 on $\bf{X}$ and $\bf{O}$ corresponding to vertex $v$.
  Identifying $\widehat{N}(v)$ and
 $\widehat{N}(v,F)$ with $N(v)$ and $N(v,F)$ respectively, we note that
 $\widehat{W}_v = U_v \times T_{\widehat{N}(v)}/ \widehat{\sim_v}$ may be identified with
 $W_v= U_v \times T_{N(v)}/{\sim_v} $. Hence $O_v = W_v/\widehat{G}_v$ and
  $f: O_v \to X_v$ is given by the projection $W_v/\widehat{G}_v \to W_v/G_v $ where
 $\widehat{G}_v = \widehat{N}/N(v) $ is a subgroup of $G_v = N/N(v)$.  So  $\bf{f}: \bf{O} \to \bf{X}$
 is in fact an orbifold covering. The deck group for this covering is clearly $N/\widehat{N}$. 

\end{proof}

\begin{theorem}\label{tofg}
The quasitoric orbifold $\bf{O}$ is the orbifold universal cover of the quasitoric orbifold $\bf{X}$.
The orbifold fundamental group $\pi_1^{\rm orb}(\mathbf{X})$ of $\bf{X}$ is
isomorphic to $N/ \widehat{N}$.
\end{theorem}

\begin{proof}
Let $\Sigma$ denote the singular loci of $\bf{X}$ (refer to definition \ref{singpt}).
 The set $\Sigma$ has real codimension at least $2$ in $X$.
 Note that $\pi(\Sigma)$ is a union of faces of $P$.
Let $P_{_\Sigma}= P - \pi(\Sigma)$.

Observe that $X - \Sigma = \pi^{-1}(P_{_\Sigma} ) = P_{_\Sigma} \times T_N /\sim $.
Since $P_{_\Sigma}$ is contractible, $\pi_1( P_{_\Sigma} \times T_N ) \cong \pi_1(T_N) \cong N $.
When we take quotient of $P_{_\Sigma} \times T_N $ by the equivalence relation $\sim$, certain elements
of this fundamental group are killed. Precisely, if $ P_{_\Sigma}$ contains a point $p$ which belongs to
the intersection of certain facets $F_{1},\ldots, F_k$ of $P$, then the elements $\lambda_1,\ldots,
\lambda_k$ of $N$ given by the corresponding characteristic vectors map to the identity element of
$\pi_1(  X - \Sigma )$. Let $I(\Sigma)$ be the collection of facets of $P$ that have nonempty intersection
with $P_{_\Sigma}$. Let $N(\Sigma)$ be the submodule generated by those $\lambda_i$ for which $i \in I(\Sigma)$.
Then the argument above suggests that $\pi_1(X-\Sigma)= N/N(\Sigma)$. Indeed, this can be established easily
by systematic use of the Seifert-Van Kampen theorem.

It is instructive to first do the proof in the case ${\bf X}$ is primitive (see Definition \ref{prim}).
Here $G_{F_i}=1$ (see \eqref{GE}) for each facet $F_i$. Hence $I(\Sigma)= I$ and $N_{\Sigma}=\widehat{N}$.
Therefore $\pi_1 (X- \Sigma) = N/\widehat{N}$. Hence by Proposition \ref{pouc},
$f_0: O - f^{-1}(\Sigma) \to X - \Sigma $, where $f_0$ is the restriction of $f$, is the universal
covering.
 Now if ${\bf p}: {\bf W} \to {\bf X}$ is any orbifold cover
then the induced map $p_0: W - p^{-1}(\Sigma) \to X - \Sigma $ is a manifold cover. Since $ p^{-1}(\Sigma)$
has real codimension at least two in $W$, $ W - p^{-1}(\Sigma)$ is connected.
By a metric completion argument it follows that $f_0 $ factors through $p_0$ and ${\bf f}$ factors through ${\bf p}$. 

For the general case we will use an argument which is similar to that of Scott \cite{[Sc]} 
for orbifold Riemann surfaces.
The underlying idea also appeared in remarks after Proposition 13.2.4 of Thurston \cite{[Th]}.

  The group $N/\widehat{N}$ is naturally a quotient of $\pi_1(X-\Sigma)= N/N(\Sigma)$
and the corresponding projection homomorphism has kernel $K =\widehat{N}/N(\Sigma)$.
 Consider the manifold covering $f_0: f^{-1}(X- \Sigma)  \to X - \Sigma $
 obtained by restricting the map $f: O \to X$.
 Note that $\pi_1(f^{-1}(X-\Sigma))=K $ and the deck group of $f_0$ is
 $N/\widehat{N} $.
 Let ${\bf W}$ be any orbifold covering of ${\bf X}$ with projection map ${\bf p}$.
 Then $W_0 = W - p^{-1}(\Sigma)$ is a
 covering of $X-\Sigma$ in the usual sense. We claim that $\pi_1(W_0) $ contains $K$ as a subgroup.

 Let $\bar{\lambda}_i$ denote the image of $\lambda_i$ in $ N/N(\Sigma)$. Obviously
 $\{ \bar{\lambda}_i, \,: \,  i \in I-I(\Sigma) \} $ generate $K$.
 Physically such a $\bar{\lambda}_i$ can be represented by
 the conjugate of a small loop $c_i$ in $X-\Sigma$ going around some point
  $x_i \in \pi^{-1}(F_i^{\circ})$ once in a plane
 transversal to $\pi^{-1}(F_i) $, where  $F_i^{\circ}$ denotes the relative interior of the facet $F_i$.
  The point $x_i$ has a neighborhood $U$ in $X$ homeomorphic to
   $\CC^{n-1} \times (\CC /G_{F_i})$. Therefore a connected component $V$ of the preimage $p^{-1}(U) \subset W$
  is homeomorphic to $ \CC^{n-1} \times (\CC/G^{\prime}_{F_i})$ where $G^{\prime}_{F_i}$ is a subgroup
  of $G_{F_i}$. We may assume, without loss of generality, that $c_i$ lies in the plane $\{0\} \times \CC/G_{F_i} $.
   By the definition of $G_{F_i}$, $\bar{\lambda}_i$ is
 trivial in $G_{F_i}$ and hence in $G^{\prime}_{F_i}$. Identifying $G_{F_i}$ with the deck group of the covering 
 $\CC^{\ast} \to \CC^{\ast}/G_{F_i}$, we infer that  $c_i$ lifts to a loop in $\CC^{\ast}$ and consequently
  in $ \CC^{\ast}/G^{\prime}_{F_i} $. Hence $c_i$ lifts
 to a loop in $V- p^{-1}(\Sigma) $. Thus each generator and therefore every element of $K$ is represented by a
 loop in $W_0$.  This induces a homomorphism $K \to \pi_1 (W_0)$. This homomorphism is injective since $K$ is a
 subgroup of the fundamental group of the space $X-{\Sigma}$ which has $W_0$ as a cover.

 For any orbifold covering ${\bf W}$ of ${\bf X}$, the associated covering $W_0$ of $X -\Sigma$ admits
 a covering by $f^{-1}(X-\Sigma) \subset O$ since $\pi_1(X-\Sigma) = K$ is a normal subgroup of $\pi_1(W_0)$.
  Thus ${\bf O}$ is an orbifold cover of ${\bf W}$. Hence ${\bf O}$ is the universal orbifold cover of ${\bf X}$
   and $N/\widehat{N}$ is the orbifold fundamental group of ${\bf X}$.
 \end{proof}

\begin{remark}\label{globalq}
Note that the orbifold fundamental group of a  quasitoric orbifold is always a finite group. It follows
that a  quasitoric orbifold is a global quotient if and only if its orbifold universal cover is a smooth
manifold. Therefore Theorem \ref{tofg} yields a rather easy method for determining if a  quasitoric orbifold
is a global quotient or not.
\end{remark}

\begin{example} If  $\widehat{N} = N$, then $\bf{X}$ is not a global quotient unless $\bf{X}$ is a
manifold. For instance, let $P$ be a $2-$dimensional simplex with characteristic vectors $(1,1), \, 
(1,-1), \, (-1,0)$ and let ${\bf X}$ be the quasitoric orbifold corresponding to this model. 
 Then $N = \widehat{N}$, but ${\bf X}$ has an orbifold singularity at $\pi^{-1}(v)$
where $v = F_1 \cap F_2$. Therefore $\bf{X}$ is not a global quotient. In fact $\bf{X}$ is
equivariantly diffeomorphic to the weighted projective space $\bf{P}(1,1,2)$.  
\end{example}

 \section{Homology with rational coefficients}\label{hom}

 Following Goresky \cite{[Gor]} one may obtain a $CW$ structure on a quasitoric
 orbifold. However it is too complicated for easy computation of homology.
  We introduce the notion of $\mathbf{q}-CW$ complex where an open cell
  is the quotient of an open disk by action of a finite group. Otherwise the
  construction mirrors the construction of usual $CW$
 complex given in Hatcher \cite{[Ha]}. We show that our $\mathbf{q} $--cellular homology of a
 $\mathbf{q}-CW$ complex is isomorphic to its singular homology  with coefficients
 in $\QQ$. We then follow the main ideas
  of the computation for the manifold case in
  \cite{[DJ]} to compute the rational homology groups of $X$.

\subsection{$\mathbf{q}$--Cellular Homology} \label{Hom}
\begin{defn}
 Let G be a finite group acting linearly, preserving orientation, on an
 $n$--dimensional disk $\bar{D}^{n}$ centered at the origin. Such an action
  preserves $S^{n-1}$.
  We call the quotient $ \bar{D}^{n}/G $ an $n$--dimensional $\mathbf{q} $--disk.
 Call $ S^{n-1}/G $ a $\mathbf{q}-$sphere. An
 $n$--dimensional $\mathbf{q}$--cell  $e^{n}_G = e^n(G)/G$ is defined to be a copy of $D^{n}/G$
 where $\bar{e}^{n}(G)$ is $G-$equivariantly homeomorphic to $\bar{D}^n$. We will denote
  the boundary of $\bar{e}^n(G)$ by $S^{n-1}$ without confusion.
\end{defn}

 Start with a discrete set $ X_{0} $ , where points are regarded as $0-$dimensional $\mathbf{q}$-cells.
 Inductively, form the $n$-dimensional $\mathbf{q}$-skeleton $ X_{n} $ from $ X_{n-1} $ by
 attaching $n$-dimensional $\mathbf{q}$-cells $ e^{n}_{G_{\alpha}}$ via continuous
 maps $ \phi_{\alpha} : S^{n-1}_{\alpha}/G_{\alpha} \to X_{n-1} $. This means that
 $ X_{n} $ is the quotient space of the disjoint union $ X_{n-1}\sqcup_{\alpha}
  \bar{e}_{G_{\alpha}}^n $ of $ X_{n-1} $ with a finite collection of $n$-dimensional
  $\mathbf{q}$-disks $ \bar{e}^{n}_{\alpha}(G_\alpha) /G_{\alpha} $ under the identification
  $ x\sim \phi_{\alpha}(x)  $  for $ x\in S^{n-1}_{\alpha}/G_{\alpha} $.

 Assume $ X=X_{n} $ for some finite $n$. The topology of $ X $ is the quotient topology
built inductively. We call a space $ X $ constructed in this way a finite $\mathbf{q}-CW$ complex.

 By Proposition 2.22 and Corollary 2.25 of \cite{[Ha]},

 \begin{equation}
  H_{p}((X_{n}, X_{n-1}); \QQ) = \bigoplus_{\alpha}
   \widetilde{H}_p \( \frac{\bar{D}_{\alpha}^n /G_{\alpha} }{S_{\alpha}^{n-1} /G_{\alpha}}; \QQ \)
  \end{equation}

 Note that
 \begin{equation}
  \widetilde{H}_p \( \frac{\bar{D}_{\alpha}^n /G_{\alpha} }{S_{\alpha}^{n-1} /G_{\alpha}}; \QQ \)
   = \left\{ \begin{array}{ll}  H_{p-1} (S_{\alpha}^{n-1} /G_{\alpha} ; \QQ )  & {\rm if} \, p\ge 2\\
                                       0 & {\rm otherwise} \end{array} \right.
 \end{equation}

 \begin{lemma}\label{qlem1}
 Let $\bar{D}^n/G$ be a $\mathbf{q}-$disk. Then $S^{n-1}/G$ is a $\QQ-$homology sphere.
 \end{lemma}

\begin{proof} $S^{n-1}$ admits a simplicial $G-$complex structure. Apply Theorem 2.4
of Bredon \cite{[Bre]} and Poincar\'e duality for orbifolds.
\end{proof}

\begin{lemma}\label{qlem}
 If $X$ is a $\mathbf{q}-CW$ complex , then

\begin{enumerate}
\item
$$
  H_{p}((X_{n}, X_{n-1}); \QQ) =\left\{ \begin{array}{ll} 0 & \mbox{for}~  p\neq n \\
 \oplus_{_{i \in I_n}} \QQ & \mbox{for}~  p=n
   \end{array} \right.
$$
where $I_n$ is the set of $n$--dimensional $\mathbf{q}$--cells in $X$.

\item $ H_{p}(X_{n};\QQ) = 0 $ for $ p>n $. In particular,
 $ H_{p}(X;\QQ) =0 $ for $p> {\rm dim} (X)$.

\item The inclusion $i : X_{n} \hookrightarrow X $ induces an isomorphism $ i_{*} : H_{p}(X_{n};\QQ) \to H_{p}(X;\QQ) $ if $ p<n $.
\end{enumerate}
\end{lemma}

\begin{proof}
Proof is similar to the proof of Lemma 2.3.4 of \cite{[Ha]}. The key ingredient is  Lemma \ref{qlem1}.
\end{proof}

 Using Lemma \ref{qlem}
  we can define $\mathbf{q}$-cellular chain complex $ (H_{p}(X_{p}, X_{p-1}), d_{p}) $
  and $\mathbf{q}$-cellular groups $ H_{p}^{\mathbf{q}-CW}(X; \QQ) $
  of $X $ in the same way as cellular chain complex  is defined in \cite{[Ha]}, page 139.

\begin{theorem}\label{cellsing}
 $ H_{p}^{\mathbf{q}-CW}(X;\QQ)\cong  H_{p}(X;\QQ) $ for all p.
\end{theorem}
\begin{proof}
 Proof is similar to the proof of Theorem 2.35 of \cite{[Ha]}.
\end{proof}

\subsection{Rational Homology of Quasitoric Orbifolds}

 Realize $P$ as a convex polytope in $\RR^n$ and choose a linear functional $\phi: \RR^n \to \RR$ which
distinguishes the vertices of $P$, as in proof of Theorem 3.1 in \cite{[DJ]}. The vertices are linearly
ordered according to ascending value of $\phi$.
We make the 1-skeleton of $P$ into a directed
graph by orienting each edge such that $\phi$ increases along it. For each vertex of $P$ define its
index $\mathfrak{f}(v)$, as the number of incident edges that point towards $v$.

  Let $F_v$ denote the smallest face of $P$ which contains the inward pointing edges incident to $v$.
  then $\dim F_v = \mathfrak{f}(v)$ and if $F^{\prime}$ is a face of $P$ with top vertex $v$ then $F^{\prime}$
  is a face of $F_v$. Let $\widehat{F}_v$ be the union of the relative interiors of those faces $ F^{\prime}$
  of $P$, $P$ included, whose top vertex is $v$.

  For each vertex $v$ put $e_v = \pi^{-1}(\widehat{F}_v)$.
  $e_v$ is a contractible subspace of $X(F_v)$ homeomorphic
  to the quotient of an open disk  $D^{2\mathfrak{f}(v)}$
   in $\RR^{2\mathfrak{f}(v)}$ by a finite group $G(v)$
   determined by the orbifold structure on $X(F_v)$ described in subsection \ref{chars}.
  $\widehat{F}_v $ is homeomorphic to the intersection of the unit disk in
    $\RR^{\mathfrak{f}(v)}$ with $\RR^{\mathfrak{f}(v)}_+$. Since the action of the group $G(v)$ is
    obtained from a combinatorial model, see subsection \ref{chars},
     $e_v$ is a $2 \mathfrak{f}(v)-$dimensional  $\mathbf{q}-$cell.

   $X$ can be given the structure of a $\mathbf{q}-CW$ complex as follows.
   Define the $k-$skeleton $X_{2k} := \bigcup_{\mathfrak{f}(v) = k} X(F_v) $ for $0 \le k \le n $. $X_{2k+1} = X_{2k}$
   for $0 \le k \le n-1$ and $X_{2n} =X$.
   $X_{2k}$ can be obtained from $X_{2k-1}$ by attaching those $\mathbf{q}-$cells $e_v$ for which
   $ \mathfrak{f}(v)= k$.
    The attaching maps are to be described.
  Let $\bf{\sim}$ be the equivalence relation such that $X(F_v) = F_v \times T_{N^{\perp}(F_v)}/{\bf{\sim}}$.
   The $\mathbf{q}-$disk $\bar{D}^{2\mathfrak{f}(v)}/G(v)$ can be identified with
  $F_v \times T_{N^{\perp}(F_v)}/{\bf{\approx}}$ where $(p,t) \approx (q,s)$ if $p=q \in F^{\prime}$
  for some face $F^{\prime}$ whose top vertex is $v$ and $(p,t) \sim (q,s)$. The attaching map
  $\phi_v: S^{2\mathfrak{f}(v) -1}/G(v) \to X_{2\mathfrak{f}(v) -1}  $ is the
  natural quotient map from  $(F_v - \widehat{F}_v) \times T_{N^{\perp}(F_v)}/{\approx} \to
   (F_v - \widehat{F}_v) \times T_{N^{\perp}(F_v)}/{\sim}$.

  $X $ is a $\mathbf{q}-CW$ complex with no odd dimensional cells and with $ \mathfrak{f}^{-1}(k) $ number of
  $2k-$dimensional $\mathbf{q}$-cells. Hence by $\mathbf{q}-$cellular homology theory

\begin{equation}
H_{p}^{ \mathbf{q}-CW}(X;\QQ)= \left\{ \begin{array}{ll} \bigoplus_{_{_{\mathfrak{f}^{-1}(p/2)}}}
\mathbb{Q} & \mbox{ if}~  p  \leq n ~ \mbox{and} ~ p~ \mbox{is even} \\
 0 & \mbox{otherwise}
\end{array} \right.
\end{equation}

Hence by Theorem \ref{cellsing}
\begin{equation}
 H_{p}(X;\QQ) = \left\{ \begin{array}{ll} \bigoplus_{_{ \mathfrak{f}^{-1}(p/2)}}\mathbb{Q} & \mbox{if} ~ p \leq n~ \mbox{and}~ p ~\mbox{is even}\\
 0 & \mbox{otherwise}
\end{array} \right.
\end{equation}

\section{Cohomology ring of quasitoric orbifolds}\label{cohom}

 Again we will modify some technical details but retain the broad framework of the argument in
 \cite{[DJ]} to get the anticipated answer. All homology and cohomology modules in this section
 will have coefficients in $\QQ$.

\subsection{Gysin sequence for $\mathbf{q}-$sphere bundle}\label{bundle}

 Let $ \rho : E \to B $ be a rank $n$ vector bundle with paracompact base space $B$.
  Restricting $ \rho $ to the space $ E_{0}$ of  nonzero vectors  in $E$, we obtain
   an associated projection map
 $\rho_{0} : E_{0}\to B $. Fix a finite group $G$ and a representation of $G$ on $\RR^n$. Such
 a representation induces a fiberwise linear action of $G$ on $E$ and $E_0$.

  Consider the two fiber bundles
   $\rho^{G} : E/G \to B$ and $\rho^{G}_{0} : E_{0}/G \to B$. There exist
  natural fiber bundle maps $ f_{1}: E\to E/G$ and $ f_{2}: E_{0} \to E_{0}/G$.
  These induce isomorphisms
   $f_{1}^{*} : H^{p}(E/G) \to H^{p}(E)$ and $f_{2}^{\ast} : H^{p}(E_{0}/G) \to H^{p}(E_0)$
   for each $p$. The second isomorphism is obtained by applying Theorem 2.4
   of \cite{[Bre]} fiberwise and then using Kunneth formula, Mayer-Vietoris sequence and
    a direct limit argument as in the proof of Thom isomorphism in \cite{[MS]}.
     The commuting diagram

 \[
\begin{CD}
E_{0} @>i_{1}>> E @>j_{1}>> (E,E_{0})\\
@Vf_{2}VV @Vf_{1}VV @VVf_{3}V\\
E_{0}/G @>i_{2}>> E/G @>j_{2}>> (E/G,E_{0}/G)
\end{CD}
\]
 induces a commuting diagram of two exact rows
$$
 \begin{CD}
\cdots \to H^{p-1}(E_{0}) @>\delta_{1}^{*}>>  H^{p}(E,E_{0}) @>j_{1}^{*}>> H^{p}(E) @>i_{1}^{*}>> H^{p}(E_{0})  \to \cdots \\
@Af_{2}^{*}AA @Af_{3}^{*}AA  @AAf_{1}^{*}A @AAf_{2}^{*}A \\
\cdots \to H^{p-1}(E_{0}/G)  @>\delta_{2}^{*}>>  H^{p}(E/G,E_{0}/G) @>j_{2}^{*}>> H^{p}(E/G) @>i_{2}^{*}>>  H^{p}(E_{0}/G) \to \cdots
\end{CD}
$$
 By the five lemma $ f_{3}^{*} $ is an isomorphism. Using the Thom isomorphism
 $ \cup u : H^{p-n} (E) \to H^{p}(E,E_{0}) $
 we get the isomorphism $\cup u_{G} : H^{p-n}(E/G) \to H^{p}(E/G, E_{0}/G) $  where $ \cup u_{G} = {f_{3}^{*}}^{-1}\circ \cup u \circ f_{1}^{*} $.
 Substituting the isomorphic module $ H^{p-n}(E/G) $ in place of $ H^{p}(E/G, E_{0}/G)$ in the second row of the
 above diagram, we obtain an exact sequence
$$
 \begin{CD}
\cdots \to H^{p-n}(E/G) @>g>> H^{p}(E/G) \to H^{p}(E_{0}/G) \to H^{p-n+1}(E/G) \to \cdots
\end{CD}
$$
 where $ g = j_2^{*} \circ \cup u_{G}$.
  The pull back of cohomology class
 $  u_{G}|(E/G) $ in $ H^{n}(B) $ by the zero section of $\rho^G$ will be called the Euler class $e$ of
  $ \rho^G $. Now substitute the isomorphic cohomology ring $ H^{*}(B) $ in place of  $ H^{*}(E/G) $
  in the above sequence.
  This yields the  Gysin exact sequence for the $\mathbf{q}-$sphere bundle
 $\rho^{G}_0 : E_0/G \to B$
\begin{equation}
 \begin{CD}
\cdots \to H^{p-n}(B;\QQ) @>\cup e>> H^{p}(B;\QQ) \to H^{p}(E_{0}/G;\QQ) \to H^{p-n+1}(B;\QQ) \to \cdots
\end{CD}
\end{equation}

\begin{remark}\label{euler}
Euler classes of $\rho : E \to B$ and $\rho^{G} : E/G \to B$ are the same since $f_1^{*}$ is an isomorphism.
\end{remark}

\subsection{A Borel construction}\label{local}

 Let $ K $ be the simplicial complex associated to the  boundary of the dual polytope of $ P $. Then $ P $ is the cone on the barycentric subdivision of $ K $.  $P$ can be split into cubes $P_{\sigma}$ where $\sigma$ varies over
  $(n-1)$--dimensional faces of $K$. These correspond bijectively to vertices of $P$. We regard the $k$--cube as the orbit space of standard $k-$dimensional torus action on the $2k$--disk
\begin{equation}
 \bar D^{2k}=\{(z_{1}, \cdots , z_{k})\in \mathbb{C}^{k} : \lvert z_{i} \rvert \leq 1\}
\end{equation}

 Define $ BP_{\sigma} = ET_{N}\times_{T_{N}}((P_{\sigma}\times T_{N})/\sim) \cong  ET_{N}\times_{T_{N}}(\bar{D}^{2n}/G_{\sigma}) $,
  where $G_{\sigma} = G_{v_{\sigma}}$, $v_{\sigma}$ being the vertex in $ P $ dual to $ \sigma $.
If $ \sigma_{1}$ is another $(n-1)$ simplex in $K$ such that $ \sigma\cap\sigma_{1} $ is an $(n-2)$ simplex
 then $ BP_{\sigma} $ and $ BP_{\sigma_{1}} $ are glued along the common part of the boundaries of $ P_{\sigma} $
 and $ P_{\sigma_{1}} $. In this way $ BP_{\sigma} $ fit together to yield $ BP = ET_{N}\times_{T_{N}}X$. Let $\mathfrak{p} : BP\rightarrow BT_{N}$ be the Borel map which is a fibration with fiber $X$. The fibration $ \mathfrak{p}: BP \to BT_{N} $ induces a homomorphism $ \mathfrak{p}^{*}:H^{*}(BT_{N};\QQ) \to H^{*}(BP;\QQ) $.

 The face ring or Stanley-Reisner ring $SR(P)$ of a polytope $P$ over $\QQ$ is the quotient of the ring $\QQ[w_1, \ldots, w_m]$,
  where the variables $w_i$ correspond to the facets of $P$,
 by the ideal $\mathcal{I}$ generated by all monomials $w_{i_{1}}\cdots w_{i_{k}}$ such that the corresponding intersection of facets
 $F_{i_{1}} \cap \ldots \cap F_{i_{k}} $ is empty. The face ring is graded by declaring the degree of each $w_i$ to be $2$.
   The following result resembles Theorem 4.8 of \cite{[DJ]}.

\begin{theorem}\label{tbp}
  Let $P$ be an $n-$polytope and $SR(P)$ be the face ring of $P$ with coefficients in $\QQ$. The map $ \mathfrak{p}^{*}: H^{*}(BT_{N};\QQ) \to H^{*}(BP;\QQ) $ is surjective and induces an isomorphism of graded rings $H^{*}(BP;\QQ) \cong SR(P)$.
\end{theorem}

\begin{proof}
  Suppose $ \sigma $ is an $(n-1)$--simplex in $K$ with vertices $ w_{1},\ldots ,w_{n} $. Note that there is a one-to-one correspondence between facets of $P$ meeting at $v_{\sigma}$ and vertices of $\sigma$.
   Let $ P_{\sigma} $ be the corresponding $n-$cube in $P$.
   Then $ BP_{\sigma}= ET_{N}\times_{T_{N}}(\bar{D}^{2n}/G_{\sigma})$ is a $\bar{D}^{2n}/G_{\sigma}$
    fiber bundle over $BT_{N}$.
    Hence $ ET_{N}\times_{T_{N}}(S^{2n-1}/G_{\sigma}) \to  BT_{N}$ give the associated $\mathbf{q}-$sphere bundle
    $\mathfrak{p}_{\sigma} : BP_{\partial\sigma}  \to BT_{N}$. Also consider the
  disk bundle $\mathfrak{r} :  ET_{N}\times_{T_{N}} \bar{D}^{2n} \to BT_{N}$.
   It is bundle homotopic to the complex vector
  bundle $\mathfrak{r}^{\prime} : ET_{N}\times_{T_{N}}\CC^{n} \to BT_{N}$.
   Since $T_{N}$ acts diagonally on $\CC^n$, the last bundle
  is the sum of line bundles $\mathcal{L}_{1}\oplus \cdots \oplus \mathcal{L}_{n}$ where
  $\mathcal{L}_j$ corresponds to $j-$th coordinate
  direction in $\CC^n$ and hence to $w_j$. Without confusion, we set
    $c_{1}(\mathcal{L}_{i}) = w_{i} \in H^{2}(BT_{N}; \QQ)$.
  By the Whitney product formula $c_{n}(\mathfrak{r}^{\prime}) = w_{1} \cdots w_{n}$.
   Hence from section \ref{bundle} the Euler class of the $\mathbf{q}-$sphere bundle $\mathfrak{p}_{\sigma}$ is $e = w_{1} \cdots w_{n} $.

 Now consider the Gysin exact sequence for $\mathbf{q}-$sphere bundles

\[
\begin{CD}
\cdots \to H^{*}(BP_{\partial\sigma}) \to H^{*}(BT_{N}) @>\cup e>>  H^{*+2n}(BT_{N}) @>\mathfrak{p}_{\sigma}^{*}>>  H^{*+2n}(BP_{\partial\sigma}) \to H^{*+2n}(BT_{N}) \to \cdots
\end{CD}
\]
Since the map $\cup e$ is injective, by exactness $\mathfrak{p}_{\sigma}^{*}$ is surjective and we get the following diagram
\begin{equation}\label{equ}
\begin{CD}
0 \to H^{*}(BT_{N}) @>\cup {e}>> H^{*+2n}(BT_{N}) @>\mathfrak{p}_{\sigma}^{*}>>   H^{*+2n}(BP_{\partial \sigma}) \to 0 \\
@V{id}VV  @V{id}VV   @. \\
\QQ[w_{1},\cdots ,w_{n}] @>w_{1}\ldots w_{n}>>  \QQ[w_{1},\ldots ,w_{n}]
\end{CD}
\end{equation}

  Hence from diagram (\ref{equ}) $H^{*}(BP_{\partial\sigma})= \QQ[w_{1},\cdots,w_{n}]/(w_{1}\ldots w_{n})$.
 Since  $\bar{D}^{2n}/G_{\sigma}$ is contractible,
  $H^{*}(BP_{\sigma};\QQ)= H^{*}(BT_{N};\QQ)= \QQ[w_{1},\cdots,w_{n}]$.
Using induction on the dimension of $K$ and an application of the
Mayer-Vietoris sequence we get the conclusion of the Theorem.
\end{proof}

Consider the Serre spectral sequence of the fibration $\mathfrak{p} : BP \to BT_{N} $ with fiber $X$. It has $E_{2}$-term
$E_{2}^{p,q} = H^{p}(BT_{N}; H^{q}(X))= H^p(BT_{N})\otimes H^q(X)$. Using the formula for Poincar\'e series of $X$
it can be proved that this spectral sequence degenerates, $E_{2}^{p,q} = E^{p,q}_{\infty}$ (see Theorem 4.12 of \cite{[DJ]}).
Let $ j : X \to BP $ be inclusion of the fiber. Then $ j^{*} : H^{*}(BP) \to H^{*}(X) $ is surjective
(see Corollary 4.13 of \cite{[DJ]}).

 We have natural identifications $H_{2}(BP) = \QQ^{m}$ and $H_{2}(BT_{N}) = \QQ^{n}$.
 Here $\QQ^{m}$ is regarded as the  $\QQ$ vector space with basis corresponding to the set of codimension
 one faces of $P$.
 $\mathfrak{p}_{*}: H_{2}(BP) \to H_{2}(BT_{N})$ is naturally identified with the characteristic map
 $\Lambda : \QQ^{m} \to \QQ^{n}$ that sends $w_i$, the $i$th standard basis vector of $\QQ^m$, to $\lambda_i$.
 The map $\mathfrak{p}^{*}: H^{2}(BT_{N}) \to H^{2}(BP)$ is then identified with the dual map
 $ \Lambda^{*}: (\QQ^{n})^{*} \to (\QQ^{m})^{*}$. Regarding the map $\Lambda $ as an $n\times m$ matrix $\lambda_{ij}$,
  the matrix for $\Lambda^{*}$ is the transpose. Column vectors of $\Lambda^{*}$ can then be regarded as linear combinations of $w_{1}, \ldots , w_{m}$.  Define
 \begin{equation}
\lambda^{i}=\lambda_{i1} w_{1}+ \cdots + \lambda_{im} w_{m}.
\end{equation}

We have a short exact sequence
 \[
\begin{CD}
0 \to H^{2}(BT_{N}) @>{\mathfrak{p}^{*}}>> H^{2}(BP) @>{j^{*}}>> H^{2}(X) \to 0\\
 @|  @| \\
 (\QQ^{n})^{*} @>{\Lambda^{*}}>> (\QQ^{m})^{*}.
\end{CD}
\]

Let $\mathcal J$ be the homogeneous ideal in $\QQ[w_{1}, \ldots, w_{m}]$ generated by the $\lambda^{i}$ and let
 $\bar{\mathcal J}$ be its image in the face ring $SR(P)$. Since $j^{*}: SR(P) \to H^{*}(X) $ is onto and
$\bar{\mathcal J}$ is in its kernel, $j^{*}$ induces a surjection $SR(P)/\bar{\mathcal J} \to H^{*}(X) $.

\begin{theorem}\label{cohomring}
Let $\mathbf{X}$ be the quasitoric orbifold associated to the
  combinatorial model $(P,N,\{\lambda_i\})$. Then $H^{*}(X; \QQ)$ is the quotient of the face ring of $P$ by
  $\bar{\mathcal J}$; i.e.,
 $H^{*}(X; \QQ) = \QQ[w_{1}, \ldots ,w_{m}]/ ( \mathcal{I} + \mathcal{J} )$.
\end{theorem}

\begin{proof}
We know that $H^{*}(BT_{N})$ is a polynomial ring on $n$ generators, and $H^{*}(BP)$ is the face ring.
 Since the spectral sequence degenerates, $H^{*}(BP)\simeq H^{*}(BT_{N})\otimes H^{*}(X)$.
 Furthermore, $\mathfrak{p}^{*} : H^{*}(BT_{N}) \to H^{*}(BP)$ is injective and $\bar{\mathcal J}$ is identified
  with the image of $\mathfrak{p}^{*}$. Thus $H^{*}(X)=H^{*}(BP)/\bar{\mathcal J}  =
   \QQ[w_{1}, \ldots ,w_{m}]/(\mathcal{I}+\mathcal{J})$.
\end{proof}

\section{Stable almost complex structure}\label{almostc}

 Buchstaber and Ray \cite{[BR]} have shown the existence of a stable almost complex
 structure on omnioriented quasitoric manifolds. We generalize their result to
 omnioriented quasitoric orbifolds (see section \ref{intro} for definition).
  Let $m$ be the cardinality of $I$, the set of facets of the polytope $P$.
  We will realize the orbifold $\bf{X}$ as the quotient
 of an open set of $\CC^m$. Consider
 the natural combinatorial model $( \RR^{m}_+ , L \cong \ZZ^m, \{e_i \} )$ for $\CC^m$.
 Let $\pi_s: \CC^m \to \RR^{m}_+ $ be the projection map corresponding to taking modulus
 coordinatewise.
 Embed the polytope $P$ in $\RR^{m}_+$ by the map $d_{\mathcal F}: P \to \RR^m $ where the $i$th coordinate of
 $d_{\mathcal F}(p)$
  is the Euclidean distance $(d(p,F_i))$ of $p$ from the hyperplane of the $i$th facet $F_i$ in
  $\RR^n$. Consider the thickening $W^{\RR}(P) \subset \RR^{m}_+$ of $d_{\mathcal F}(P)$, defined by
  \begin{equation}
  W^{\RR}(P)= \{ f:I \to \RR_+ | f^{-1} (0) \in \mathfrak{L}_F(P)  \}
 \end{equation}
where $\mathfrak{L}_F(P)$ denotes the face lattice of $P$.

Denote the $n$--dimensional linear subspace of $\RR^m$ parallel to $d_{\mathcal F}(P)$ by $V_P$ and
its orthogonal complement by $V_P^{\perp}$. As a manifold with corners, $W^{\RR}(P)$ is canonically
diffeomorphic to the cartesian product $d_{\mathcal F}(P) \times \exp(V_P^{\perp})$ (see \cite{[BR]}, Proposition 3.4).

 Define the spaces $W(P)$ and ${\mathcal Z}(P)$ as follows.
 \begin{equation}\label{defwp}
 W(P):= \pi_s^{-1}(W^{\RR}(P)), \quad \quad  {\mathcal Z}(P) :=
\pi_s^{-1}(d_{\mathcal F}(P) ).
 \end{equation}
  $W(P)$ is an open subset of $\CC^m$ and there is a canonical diffeomorphism
   \begin{equation}\label{wpdiffeo}
   W(P) \cong {\mathcal Z}(P) \times \exp(V_P^{\perp})
   \end{equation}

   Let $\Lambda : L  \to N  $ be the map of $\mathbb{Z}-$modules
 which maps the standard generator $e_i$ of $L$ to the dicharacteristic vector $\lambda_i$.
 Let $K$ denote the kernel of this map.  Recall the submodule $ \widehat{N}$ of $N$ generated
 by the dicharacteristic vectors and the orbifold universal cover $\bf O$ from section \ref{ofg}.
  Since
 the $\mathbb{Z}-$modules
 $L$ and $\widehat{N}$ are free, the sequence
 \begin{equation}\label{split}
 0 \longrightarrow K \longrightarrow L \stackrel{\Lambda}{\longrightarrow} \widehat{N} \longrightarrow 0
 \end{equation}
 splits and we may write $L = K \oplus \widehat{N} $. Hence $K_R \cap L = K$ and applying the second isomorphism
 theorem for groups we can consider the torus $T_K := K_R/K $ to be a subgroup of $T_L$.
  In fact we get a split exact sequence
 \begin{equation}\label{split2}
 1 \longrightarrow T_K \longrightarrow T_L \stackrel{\Lambda_{\ast}}{\longrightarrow} T_{\widehat N}
  \longrightarrow 1
 \end{equation}

 For any face $F$ of $P$ let $L(F)$ be the sublattice of $L$ generated by the basis vectors $e_i$ such that
 $d_{\mathcal F}(F)$ intersects the $i$th facet of $\RR^m_+$, that is the coordinate hyperplane $\{x_i = 0\}$.
 Note that image of $L(F)$ under $\Lambda$ is precisely $\widehat{N}(F)$, so that the preimage
  $\Lambda^{-1}(\widehat{N}(F))= K\cdot L(F)$.
 Consider the exact sequence
 \begin{equation}\label{split3}
 0 \longrightarrow \frac{K\cdot L(F)}{L(F)} \longrightarrow \frac{L}{L(F)} \stackrel{\Lambda}{\longrightarrow}
 \frac{\widehat N}{\widehat{N}(F)} \longrightarrow 0
 \end{equation}

 Since the dicharacteristic vectors corresponding to the facets whose intersection is $F$ are linearly
 independent, it follows from the definition of $K$ and $\Lambda$ that $K \cap L(F) = \{0\}$. Hence
 by the second isomorphism theorem we have a canonical isomorphism
  \begin{equation}\label{split4}
  \frac{K\cdot L(F)}{L(F)} \cong K
 \end{equation}

 So \eqref{split3} yields
 \begin{equation}\label{split5}
 0 \longrightarrow K \longrightarrow \frac{L}{L(F)} \stackrel{\Lambda}{\longrightarrow}
 \frac{\widehat N}{\widehat{N}(F)} \longrightarrow 0
 \end{equation}

 In general $\displaystyle{\frac{\widehat N}{\widehat{N}(F)}}$ is not a free $\ZZ-$module. Let
 $\widehat{N}^{\prime}(F) = (\widehat{N}(F)\otimes_{\ZZ} \QQ) \cap \widehat{N}$. Define
 \begin{equation}
 \Lambda^{\prime} = \Lambda \circ \phi
 \end{equation}
  where $\phi$ is the canonical projection
  \begin{equation}
  \phi: \frac{\widehat N}{\widehat{N}(F)} \longrightarrow \frac{\widehat N}{\widehat{N}^{\prime}(F)}.
  \end{equation}

 Since $\displaystyle{\frac{\widehat N}{\widehat{N}^{\prime}(F)}}$ is
 free, the following exact sequence splits
 \begin{equation}\label{split6}
 0 \longrightarrow \frac{\Lambda^{-1}(\widehat{N}^{\prime}(F))}{L(F)} \longrightarrow \frac{L}{L(F)} \stackrel{\Lambda^{\prime}}{\longrightarrow}
 \frac{\widehat N}{\widehat{N}^{\prime}(F)} \longrightarrow 0
 \end{equation}

 Denoting the modules in \eqref{split6} by $\bar{K}$, $\bar{L}$ and $\bar{N}$ respectively we
 obtain a split exact sequence of tori
 \begin{equation}\label{split7}
 0 \longrightarrow T_{\bar{K}} \stackrel{\theta_1}{\longrightarrow} T_{\bar{L}}
  \stackrel{\Lambda^{\prime}_{\ast}}{\longrightarrow}
  T_{\bar{N}} \longrightarrow 0
 \end{equation}

 Note that $ K $ is a submodule of same rank of the free module $ \bar{K}$ and there is a
 natural exact sequence
 \begin{equation}\label{split8}
 0 \longrightarrow \frac{\widehat{N}^{\prime}(F)}{\widehat{N}(F)} \longrightarrow T_{K}
  \stackrel{\theta_2}{\longrightarrow} T_{\bar{K}} \longrightarrow 0
 \end{equation}

 The composition
 \begin{equation}
 \theta_1 \circ \theta_2 : T_{K} \longrightarrow T_{\bar{L}}
 \end{equation}
 defines a natural action of $T_{K}$ on $T_{\bar{L}}$ with isotropy $\widehat{G}_F =
 \widehat{N}^{\prime}(F)/\widehat{N}(F)$ and
 quotient $T_{\bar{N}}$.

 Since $T_{\bar{N}}$ is the fiber of $\widehat{\pi}: O \to P$ and $T_{\bar{L}}$ is the fiber of $\pi_s: {\mathcal Z}(P)
 \to P $ over any point in the relative interior of the arbitrary face $F$, it follows $O$ is quotient of
  ${\mathcal Z}(P)$ by the above action of $T_{K}$. This action of $T_K$ is same as the restriction of
  its action on $\CC^m$ as a subgroup of $T_L$ and hence $\CC_{\times}^m$.
   By \eqref{wpdiffeo} it follows that $O$ is the quotient
  of the open set $W(P)$ in $\CC^m$ by the action of the subgroup $T_K \times \exp(V_P^{\perp}) $ of
  $\CC_{\times}^m$,
  \begin{equation}
  O = \frac{W(P)}{T_K \times \exp(V_P^{\perp})}.
   \end{equation}

  The induced action of $\widehat{H} := T_K \times \exp(V_P^{\perp})$
   on the real tangent bundle $\mathcal{T}W(P)$ of $W(P)$
   commutes with the almost complex structure $J: \mathcal{T}W(P) \to \mathcal{T}W(P)$
    obtained by restriction of the
 standard almost complex structure on $\mathcal{T}\CC^m$. Therefore the quotient
  $\widehat{\mathfrak{W}}$ of $\mathcal{T}W(P)$ by $\widehat{H}$
 has the structure of an almost complex orbibundle (or orbifold vector bundle) over $\bf{O}$. Moreover this
 quotient splits, by an Atiyah sequence, as the direct sum of a trivial rank $2(n-m)$ real bundle
 $\widehat{\mathfrak h}$
 over $O$ corresponding to the  Lie algebra of $\widehat H$ and the orbifold
 tangent bundle $\mathcal{T}\bf{O}$ of $\bf{O}$.
 The existence of a stable almost complex structure on $\mathcal{T}\bf{O}$ is thus established.


 $\mathcal{T}\CC^m$ splits naturally into a direct sum of $m$ complex line bundles corresponding to the complex
 coordinate directions which of course correspond to the facets of $P$. We get a corresponding splitting
 $\mathcal{T}W(P) = \oplus C_F$. The bundles $C_F$ are invariant under $J$ as well $\widehat H$.
  Therefore the quotient of $C_F$ by $\widehat{H}$ is a complex orbibundle $\widehat{\nu}(F)$ of rank one on $\bf{O}$ and $\widehat{\mathfrak W}= \oplus \widehat{\nu}(F)$.

 It is not hard to see that the natural action of $T_{\widehat N}$ on $\widehat{\mathfrak W}$ commutes
  with the almost complex structure
 on it.  The quotient ${\mathfrak W}:=\widehat{\mathfrak W}/(N/\widehat{N})$ is an orbibundle on $X$ with
 an induced almost complex structure since $(N/\widehat{N})$ is a subgroup of $T_{\widehat N}$.
 Furthermore $\mathcal{T}{\bf X}$ is the quotient of $\mathcal{T}{\bf O}$ by  $N/\widehat{N}$.
 Therefore ${\mathfrak W} = \mathcal{T}{\bf X} \oplus \mathfrak{h}$ where $\mathfrak h$ is the quotient
 of $\widehat{\mathfrak h}$ by $N/\widehat{N}$. Since the action of $T_{\widehat N}$
 and hence $N/\widehat{N}$ on $\widehat{\mathfrak h}$ is trivial, $\mathfrak h$ is a trivial vector bundle
 on $X$. Hence the almost complex structure on $\mathfrak W$ induces a stable almost complex structure on
 $\mathcal{T}{\bf X}$.   We also have a decomposition $\mathfrak{W}= \oplus \nu(F)$ where the orbifold
 line bundle $\nu(F):= \widehat{\nu}(F)/(N/\widehat{N})$.

\subsection{Line bundles and cohomology}\label{lbc}

 Recall the manifold $\mathcal{Z}(P)$
 of dimension $m+n$ defined in equation \eqref{defwp}.
 Let $ B_L P = ET_{L} \times_{T_{L}} \mathcal{Z}(P)$.
 Since $O = \mathcal{Z}(P) /T_{K}$, $B_L P =  ET_{L} \times_{T_{L}} \mathcal{Z}(P) =
  ET_L\times_{T_K}\mathcal{Z}(P)/(T_{L}/T_K) = ET_L \times (\mathcal{Z}(P)/T_K)/(T_{\widehat{N}})
   \simeq ET_{\widehat{N}} \times_{T_{\widehat{N}}} O
   = ET_{\widehat{N}} \times_{T_{N}} O/(N/\widehat{N}) \simeq ET_N \times_{T_N} X  = BP$.

 Let $ w_1, \ldots ,w_m  $ be the generators of $H^2(BP)$ as
  in subsection \ref{local} and let
$F_i$ denote the  facet of $P$ corresponding to $w_i$. Let $\alpha_{i} : T_L \to T^1$
be the projection onto the $i$-th factor and $\CC(\alpha_{i})$ denote the corresponding $1$--dimensional representation space of $T_L$. Define
$L_i = ET_L \times_{T_L}\widetilde {L}_{i}$, where $\widetilde {L}_{i}= \CC(\alpha_{i})\times \mathcal{Z}(P)$ is the trivial equivariant line bundle over
$\mathcal{Z}(P)$.
Then $L_i$ is an orbifold line bundle over $BP$.
Let $c_{1}(L_i)$ be the first Chern class of $L_i$ in $H^{2}(BP;\QQ)$. We will show that $c_{1}(L_i) = w_i$.

Since the i-th factor of $T_L$ acts freely on $\mathcal{Z}(P) - \pi_s^{-1}(F_i)$, the restriction of $L_i$ to $BP - BF_i $ is trivial. Consider the following commutative
diagram
\[
\begin{CD}
\iota^{*}(L_i) @>>> L_i \\
@VVV  @VVV \\
(BP - BF_i) @>\iota>> BP
\end{CD}
\]
where $\iota$ is inclusion map. By naturality $c_{1}(\iota^{*}(L_i)) = \iota^{*}(c_{1}(L_i))$.
Since the bundle $\iota^{*}(L_i)$ over $BP - BF_i$ is trivial
$ \iota^{*}(c_{1}(L_i)) =  c_{1}(\iota^{*}(L_i)) = 0$.   It is easy to show that
 $B(P-F_i)=ET_L\times_{T_L}(\pi_s^{-1}(P-F_i))\simeq BP-BF_i$.
 From the proof of Theorem \ref{tbp}
 it is evident that $H^{*}(BP-BF_i;\QQ ) \cong SR(P-F_i)$.
  Hence $H^{2}(BP-BF_i)= \bigoplus_{j \neq i}  \QQ w_{j} $. $\iota^{*} : H^{2}(BP;\QQ)
 \to H^{2}(BP-BF_i; \QQ)$ is a surjective homomorphism with kernel $\QQ w_i$
  implying $c_{1}(L_i) \in \QQ w_i$.
  Naturality axiom ensures, as follows, that $c_{1}(L_i)$ is nonzero, so that
   we can identify $c_{1}(L_i)$ with $w_i$.

 Let $F$ be an edge in $F_{i}$. Then $ BF := ET_L \times_{T_L}(\pi_s^{-1}(F))
 \simeq ET_N \times_{T_N}(\pi^{-1}(F))
 = ( ET_N \times_{T_F} \pi^{-1}(F))/(T_N/T_F)
 =  ( ET_N \times (\pi^{-1}(F)/T_F) )/(T_N/T_F)
 \simeq E(T_N/T_F) \times_{T_N/T_F} \pi^{-1}(F)
 \simeq ES^1 \times_{S^1} S^2 $, where $T_F$ is the isotropy subgroup
 of $F$ in $T_N$ and action of $S^1$ on $S^2$ is corresponding action
  of $T_N/T_F$ on $\pi^{-1}(F)$.
 Let $L_{i}(F)$ is the pullback of orbibundle $L_i$. Using Thom isomorphism and cohomology exact
 sequence obtained from
\[
\begin{CD}
BF @>s>> L_{i}(F) \to (L_i(F), BF) \simeq S^2
\end{CD}
\]
 where $s$ is zero section of $L_i$ bundle, we can show $c_1(L_{i}(F))$ is nonzero.
 Since $c_1(L_{i}(F))$ is pullback of $c_1(L_i)$, $c_1(L_i)$ is nonzero. Hence $c_1(L_i) = w_i$.

 Note  that if $F_i$ is the facet of $P$ corresponding to $L_i$,
  $L_i = ET_L \times_{T_L} \widetilde{L}_i = ET_{\widehat N} \times_{T_{\widehat N}}
 (\widetilde{L}_i/ T_K )  = ET_{\widehat N} \times_{T_{\widehat N}} \widehat{\nu}(F_i)
  = ET_{\widehat N} \times_{T_N} \widehat{\nu}(F_i)/(N/{\widehat N})
  \simeq ET_N \times_{T_N} \nu(F_i)$. Let $\mathfrak{j} : \nu(F_i) \hookrightarrow L_i$ be the
  inclusion of fiber covering $j: X \hookrightarrow BP$. Then $\mathfrak{j}^* (L_i) = \nu (F_i)$.
  Hence $ c_1 (\nu(F_i)) = j^* c_1(L_i) = j^* w_i $. Hence by Theorem \ref{cohomring} the first
  Chern classes of the bundles $\nu(F_i)$ generate the cohomology ring of $X$. We also obtain the
  formula for the total Chern class of ${\mathcal T}{\bf X}$ with the stable almost complex structure
  determined by the given dicharacteristic.
  \begin{equation}
  c(\mathcal{T}) = \prod_{i=1}^m (1 + c_1(\nu(F_i)))
   \end{equation}

\subsection{Chern numbers}\label{hg}

Chern numbers of an omnioriented quasitoric orbifold, with the induced stable almost complex structure,
 can be computed using standard localization formulae, given for instance in Chapter 9 of \cite{[CK]}.
The fixed points of the $T_N$ action correspond to the vertices of $P^n$.
While computing the numerator contributions at a vertex, one needs to recall that $T_N$ action
on the bundle $\mathfrak{h}$ is trivial.
We will give a formula for the top Chern number below. In the manifold case similar formula was
obtained by Panov in \cite{[Pan]}. In principle any Hirzebruch genus associated to a
series may be computed similarly.

Fix an orientation for $\bf{X}$ by choosing orientations for $P^n \subset \RR^n$ and $N$.
We order the facets or equivalently the dicharacteristic vectors at each vertex in a
compatible manner as follows. Suppose the vertex $v$ of $P^n$ is the intersection of facets
$F_{i_1}, \ldots, F_{i_n}$. To each of these facets $F_{i_k}$ assign the unique edge
$E_k$ of $P^n$ such that $F_{i_k} \cap E_k = v$. Let $e_k$ be a vector along $E_k$
 with origin at $v$. Then $e_1, \ldots, e_k$ is a basis of $\RR^n$ which is oriented depending
 on the ordering of the facets. We will assume the ordering $F_{i_1}, \ldots, F_{i_n}$ to be
 such that   $e_1, \ldots, e_k$ is positively oriented.

 For each vertex $v$, let $\Lambda_{(v)}$ be the matrix
 $\Lambda_{(v)} = [\lambda_{i_1} \ldots \lambda_{i_n} ]$ whose columns are ordered as described
 above. Let $\sigma(v):= {\rm det} \Lambda_{(v)}$. Then we obtain the following formula for the
 top Chern number,
 \begin{equation}
 c_n( {\mathbf X} ) = {\Sigma}_{v}  \frac{1}{\sigma (v)}
 \end{equation}

\begin{remark}
If the stable almost complex structure of an omnioriented quasitoric orbifold admits a
reduction to an almost complex structure, then $\sigma (v)$ is positive for each vertex $v$.
This follows from comparing orientations, taking $\bf{X}$ to be oriented according to the almost
complex structure.
The converse is true in the case of quasitoric manifolds, see subsection 5.4.2 of \cite {[BP]}.
 The orbifold case remains unsolved at the moment.
\end{remark}

\subsection{Chen-Ruan cohomology groups} We refer the reader to \cite{[CR], [ALR]} for definition and
motivation of the Chen-Ruan cohomology groups of an almost complex orbifold. They may be thought of
as a receptacle for a suitable equivariant Chern character isomorphism from orbifold or equivariant
 K-theory with complex coefficients, see Theorem 3.12 of \cite{[ALR]}.
Briefly, the Chen-Ruan cohomology (with coefficients in $\QQ$ or $\CC$) is the direct sum of the cohomology
of the underlying space and the cohomology of certain subspaces of it called {\it twisted sectors} which
are counted with multiplicities and rational degree shifts depending on the orbifold structure. The 
multiplicities depend on the number of conjugacy classes in the local groups and the degree shifts 
are related to eigenvalues of the linearlized local actions.
 The verification of the statements below is straightforward and left to the interested reader.

For an almost complex quasitoric orbifold $\bf{X}$, each twisted sector is a $T_N$--invariant subspace
$X(F)$ as described in subsection \ref{chars}. The contribution of $X(F)$ is counted with  multiplicity
 one less than the order of the group $G_F$, corresponding to the nontrivial elements of $G_F$.
 However
 the degree shift of these contributions depend on the particular element of $G_F$ to which the twisted
 sector corresponds. If $g = (a + N(F)) \in G_F$ where $a \in N^{*}(F)$, then the degree shift $2 \iota(g)$
 can be calculated as follows. Suppose $\lambda_1, \ldots, \lambda_k$ is the defining basis of $N(F)$. Then
 $a$ can be uniquely expressed as
 $a= \sum_{i=1}^k q_i \lambda_i$ where each $q_i$ is a rational number in the interval $[0,1)$, and
  $\iota(g)= \sum_{i=1}^k q_i$. Note that the rational homology and hence rational cohomology of
   $X(F)$ can be computed using its combinatorial model given in subsection \ref{chars}.

 Recall from subsection \ref{lbc} that if $N = \widehat{N}$ then $\bf{X}$ is the quotient
  of the manifold $\mathcal{Z}(P)$ by the group $T_K$. In this case the $T_K-$bundles $\widetilde{L}_i$
  over $\mathcal{Z}(P)$ generate the complex orbifold $K-$ring of $\bf{X}$. The images of their tensor powers
  under the equivariant
  Chern character map generate the Chen-Ruan cohomology of $\bf{X}$. These follow since the restrictions of
  the bundles $L_i$ to the subspaces $X(F)$ generate the cohomology ring of $X(F)$.

{\bf ACKNOWLEDGEMENT.} The authors would like to thank Alejandro Adem, Mahuya Datta,
 Goutam Mukherjee and Debasis Sen for helpful suggestions. The first author also thanks
 B. Doug Park, Parameswaran Sankaran and Sashi Mohan Srivastava for stimulating discussions.
  He would also like to record his gratitude to CMI
 and TIFR, institutes where part of the work was done, for their support and hospitality.

\renewcommand{\refname}{References}

\vspace{1cm}

\vfill
\end{document}